\documentclass{amsart}
\usepackage{amscd}
\usepackage{amssymb}
\usepackage[all, knot]{xy}
\usepackage{fullpage}
\usepackage{mathtools}
\xyoption{all}
\xyoption{arc}
\usepackage{color}
\usepackage{hyperref}

\def\cube#1#2#3#4#5#6#7#8{
& #5 \ar[rr] \ar[dl] \ar@{-}[d] && #6 \ar[dd] \ar[dl] \\
#1 \ar[rr] \ar[dd]  & \ar[d] & #2 \ar[dd] \\
& #7 \ar@{-}[r] \ar[dl] & \ar[r] & #8 \ar[dl] \\
#3 \ar[rr] && #4 \\
}

\def\sing{{\text{sing}}}

\def\ker{\operatorname{ker}}

\def\cC{\mathcal C}

\def\cF{\mathcal F}

\def\cT{\mathcal T}

\def\coker{\operatorname{coker}}

\def\rank{\operatorname{rank}}

\def\Proj{\operatorname{Proj}}
\def\uHom{\operatorname{\underline{Hom}}}

\def\Ext{\operatorname{Ext}}

\def\sExt{\sExt}

\def\into{\hookrightarrow}

\def\a{\alpha}
\def\b{\beta}

\newcommand{\Z}{\mathbb{Z}}

\newcommand{\fm}{{\mathfrak m}}

\numberwithin{equation}{section}

\theoremstyle{plain} 
\newtheorem{thm}[equation]{Theorem}
\newtheorem{thm-conj}[equation]{Theorem-Conjecture}
\newtheorem{defn-conj}[equation]{Definition-Conjecture}

\newtheorem*{introthm*}{Theorem}

\newtheorem{prop}[equation]{Proposition}

\newtheorem{conj}[equation]{Conjecture}

\theoremstyle{definition}

\newtheorem{ex}[equation]{Example}

\theoremstyle{remark}
\newtheorem{rem}[equation]{Remark}

\newtheorem{setup}[equation]{Setup}

\def\Dsg{\on{D}^{\on{\sing}}_{\gr}}
\def\Dsing{\Dsg}
\def\Dbgr{\on{D}^{\on{b}}_{\gr}}
\def\Db{\on{D}^{\on{b}}}

\def\Perf{\operatorname{Perf}}

\newcommand{\Hom}{\operatorname{Hom}}
\newcommand\depth{\operatorname{depth}}

\newcommand{\xra}[1]{\xrightarrow{#1}}
\newcommand{\xla}[1]{\xleftarrow{#1}}

\newcommand{\id}{\operatorname{id}}

\def\and{ \text{ and } }

\def\OO{\mathcal{O}}
\def\O{\Omega}

\def\on{\operatorname}

\def\mfgr{\on{mf}_{\on{gr}}}
\def\hmfgr{\on{hmf}_{\on{gr}}}

\def\sing{\on{sing}}

\def\MR#1{}

\def\ce{\coloneqq}

\def\PP{\mathbb{P}}
\def\gr{\on{gr}}
\def\c{\colon}
\def\th{\on{th}}
\def\T{\mathcal{T}}

\def\MCM{\underline{\on{MCM}}_{\gr}}
\newcommand{\co}{\colon}

\makeatletter
\@namedef{subjclassname@2020}{%
  \textup{2020} Mathematics Subject Classification}
\makeatother

\title{Ranks of matrix factorizations and sheaf cohomology}
\author{Michael K. Brown}
\author{Mark E. Walker}
\date{}

\vspace{18mm}  

\thanks{Brown was partially supported by NSF grant DMS-2302373 and Walker by NSF grant DMS-2200732.}

\begin{document}

\begin{abstract}
Buchweitz-Greuel-Schreyer conjectured in 1987 a lower bound on the ranks of matrix factorizations over certain local hypersurface rings. We study a graded version of this conjecture, and we show that it implies a novel conjecture concerning the cohomology of sheaves over non-Fano projective hypersurfaces.
 \end{abstract}
\maketitle
\section{Introduction}

Let $Q$ be a commutative ring and $f \in Q$. A \emph{matrix factorization of $f$} is a tuple $(F^0, F^1, s^0, s^1)$, where $F^0$, $F^1$ are finitely generated free $Q$-modules, and $s^0 \c F^0 \to F^1$, $s^1 \c F^1 \to F^0$ are $Q$-linear maps such that $s^1s^0 = f \cdot \id_{F^0}$ and $s^0s^1 = f \cdot \id_{F^1}$. 
A \emph{trivial} matrix factorization is a direct sum of copies of $(Q, Q, 1, f)$ and $(Q, Q, f, 1)$. The following is a weak version of a conjecture of Buchweitz-Greuel-Schreyer:

\begin{conj}[Weakening of \cite{BGS} Conjecture A]
\label{BGSconj}
Let $Q$ be a regular local ring of dimension $n+1$ with algebraically closed residue field, and let $e$ denote the floor of $\frac{n}{2}$, which we write as $\lfloor \frac{n}{2} \rfloor$. Suppose $f \in Q$ is  irreducible and the hypersurface $Q/(f)$ has an isolated singularity. Any nontrivial matrix factorization $(F^0, F^1, s^0, s^1)$ of $f$ satisfies $\rank(F^0) \ge~2^e$. 
\end{conj}

We explain how Conjecture~\ref{BGSconj} is a weakening of \cite[Conjecture A]{BGS} in Remark~\ref{BGSrem}. The case of Conjecture~\ref{BGSconj} where $f$ is quadratic follows from 
Kn\"orrer periodicity \cite{knorrer}. Conjecture~\ref{BGSconj} was also recently proven to hold generically by Erman~\cite{erman}. Aside from these and several additional scattered cases, the conjecture remains open.

A graded version of
Buchweitz-Greuel-Schreyer's Conjecture implies a Horrocks-type splitting criterion for vector bundles of sufficiently small rank on projective hypersurfaces: see \cite[Conjecture B]{BGS}. In this paper, we pose a quite different conjecture concerning the cohomology of sheaves on projective hypersurfaces, which is also motivated by \cite[Conjecture A]{BGS}. Before stating it, we introduce some notation: given a proper scheme $Y$ over a field $k$ and a bounded complex $\cC$ of coherent sheaves on $Y$, we set
$$
h^j(Y, \cC) \ce \dim_k \mathbf{R}^j\Gamma(Y, \cC) \quad \text{and} \quad h(Y, \cC) \ce \sum_{j \in \Z} h^j(Y, \cC).
$$
Let $i \co Y \into \PP^n_k$ be a closed embedding. Given a bounded complex $\cC$ of coherent sheaves on $Y$, we define
\begin{equation}
\label{eqn:rho}
\rho(\cC) \ce \sum_{r = 0}^n h(\PP^n_k, i_*(\cC) \otimes \Omega_{\PP^n_k}^r(r)).
\end{equation}
We caution the reader that $\rho(\cC)$ depends not just on the object $\cC$ but also on the embedding $i \co Y \into \PP^n_k$.

\begin{conj}
\label{newconj}
Let $k$ be an algebraically closed field and $i \c X = V(f)  \into \PP^n_k$ a closed embedding of a smooth irreducible hypersurface of degree $d$ such that $a \ce n+1 - d \le 0$ (i.e. such that $X$ is not Fano). For any nonzero object $\cC \in \Db(X)$, we have $\rho(\cC) \ge 2^{e+1}$, where $e \ce \lfloor \frac{n}{2} \rfloor$.
\end{conj}

The lower bound in Conjecture~\ref{newconj} is sharp: see Example~\ref{ex:sharp}. Moreover, the statement in Conjecture~\ref{newconj} is false without the assumption that $a \le 0$: see Remark~\ref{rem:false}.  Conjecture~\ref{newconj} predicts that the total rank of the $E_1$ page of the Beilinson spectral sequence associated to $\cC$ is at least $2^{e+1}$: see Remark~\ref{beilinsonremark} for details. 

\medskip
The connection between Conjectures~\ref{BGSconj} and~\ref{newconj} is illustrated by the following result:

\begin{thm}
\label{introthm}
Let $k$ be an algebraically closed field and $i \c X = V(f)  \into \PP^n_k$ a closed embedding of a smooth irreducible hypersurface of degree $d$ such that $a \ce n+1 - d \le 0$ (i.e. such that $X$ is not Fano).
\begin{enumerate}
\item Assume the graded version of Conjecture~\ref{BGSconj} holds for $f$. That is, assume all nontrivial graded matrix factorizations of $f$ (see Section~\ref{gradedmf} for the definition of a graded matrix factorization) satisfy the inequality in Conjecture~\ref{BGSconj}. In this case, Conjecture~\ref{newconj} holds for $X$. 
\item Assume $a = 0$, i.e. $X$ is Calabi-Yau. If Conjecture~\ref{newconj} holds for $X$, then the graded version of Conjecture~\ref{BGSconj} holds for $f$.
\end{enumerate}
\end{thm}

See Theorem~\ref{translation} for a more precise result concerning the relationship between ranks of matrix factorizations and sheaf cohomology, which we use to prove Theorem~\ref{introthm}. Theorem~\ref{translation} is an application of a result of Orlov relating graded singularity categories to derived categories of sheaves \cite[Theorem 2.5]{orlovcoh}. See also work of Pavlov~\cite{pavlov20}, which applies this same theorem of Orlov to give a cohomological formula for the Betti numbers of maximal Cohen-Macaulay modules over the coordinate ring of a Calabi-Yau variety in the case where the coordinate ring is Koszul. 

\subsection*{Conventions} We index cohomologically. Given a complex $(C, d_C)$ and $i \in \Z$, the shift $C[i]$ is given by $C[i]^j = C^{i + j}$ and $d_{C[i]} = (-1)^id_C$. If $Q = \bigoplus_{i \in \Z} Q_i$ is a graded ring, and a $M = \bigoplus_{i \in \Z} M_i$ is a graded $Q$-module, then the $i^{\th}$ grading twist $M(i)$ of $M$ is given by $M(i)_j = M_{i + j}$. Given graded $Q$-modules $M$ and $N$, the set of degree 0 maps from $M$ to $N$ is denoted $\Hom_R(M, N)$, and we write the graded $R$-module given by the internal $\Hom$ from $M$ to $N$ as $\uHom_R(M, N)$. The derived versions of these objects are denoted by  $\mathbf{R}$$\Hom_R(M, N)$ and $\mathbf{R}$$\uHom_R(M, N)$.  

\begin{setup}
\label{notation}
We will use the following notation throughout the paper. Let $k$ be a field, $S \ce k[x_0, \dots, x_n]$ a standard graded polynomial ring, $f \in S$ a homogeneous and irreducible polynomial of degree $d$, $R \ce S / (f)$, $\fm$ the homogeneous maximal ideal of $R$, $X$ the projective hypersurface $\Proj(R)$, and $i$ the closed embedding $X \into \PP^n_k$. We assume $n \ge 1$ throughout. We write $a \ce n+1 - d$ and $e \ce \lfloor \frac{n}{2} \rfloor$. 
\end{setup}

\section{Background}

Let $\Dbgr(R)$ denote the bounded derived category of graded $S$-modules, $\Perf_{\gr}(R) \subseteq \Dbgr(R)$ the subcategory of perfect complexes, and $\Dsg(R)$ the \emph{graded singularity category of $R$}, i.e. the Verdier quotient $\Dbgr(R) / \Perf_{\gr}(R)$. The graded singularity category can be interpreted in terms of graded matrix factorizations: let us recall some background on these objects.

\subsection{Background on graded matrix factorizations}
\label{gradedmf}

A \emph{graded matrix factorization of $f$} is a tuple $(F^0, F^1, s^0, s^1)$, where $F^0$, $F^1$ are finitely generated graded free $S$-modules, and
$$
s^0 \c F^0 \to F^1, \quad s^1 \c F^1(-d) \to F^0
$$
are homogeneous maps of degree 0 such that $s^1s^0 = f \cdot \id_{F^0}$ and $s^0s^1 = f \cdot \id_{F^1}$. We have $\rank(F^0) = \rank(F^1)$ \cite[Corollary 5.4]{eisenbud}. A \emph{morphism of graded matrix factorizations}
$
\a \c (F^0, F^1, s^0, s^1) \to (G^0, G^1, t^0, t^1)
$
 is a pair of degree 0 maps $\a^0 \c F^0 \to G^0, \a^1 \c F^1 \to G^1$  making the diagram
$$
\xymatrix{
F^1(-d) \ar[r]^-{s^1} \ar[d]^-{\a^1} & F^0 \ar[r]^-{s^0} \ar[d]^-{\a^0} & F^1 \ar[d]^-{\a^1} \\
G^1(-d) \ar[r]^-{t^1} & G^0 \ar[r]^-{t^0} & G^1 \\
}
$$
commute. A \emph{homotopy} between such morphisms $\a, \b$ is a pair of homogeneous maps $h^0 \c F^0 \to G^1(-d)$, $h^1 \c F^1 \to G^0$ satisfying the evident relations. The category $\mfgr(f)$ has objects given by graded matrix factorizations and morphisms as above, and the homotopy category $\hmfgr(f)$ of graded matrix factorizations has the same objects as $\mfgr(f)$ and morphisms given by modding out the maps in $\mfgr(f)$ by homotopy.

The category $\hmfgr(f)$ may be equipped canonically with the structure of a triangulated category; for instance, its shift functor sends $F = (F^0, F^1, s^0, s^1)$ to $F[1] \ce (F^1, F^0(d), -s^1, -s^0)$. Given $i \in \Z$, the \emph{$i^{\th}$~grading twist} of $F = (F^0, F^1, s^0, s^1)$ is $F(i) = (F^0(i), F^1(i), s^0, s^1)$.  Observe that 
$
F[2] = F(d).
$

A finitely generated $R$-module $M$ is called \emph{maximal Cohen-Macaulay} (MCM) if $\depth(M) = \dim(M)$. Let $\MCM(R)$ denote the \emph{stable category of graded MCM $R$-modules}, i.e. the category with objects given by graded MCM $R$-modules and morphisms given by $R$-linear maps modulo those that factor through a projective module. The category $\MCM(R)$ may be equipped with a triangulated structure such that the shift functor sends an MCM module $M$ to its first syzygy. It follows from (a graded version of) a theorem of Eisenbud~\cite[Theorem 6.3]{eisenbud} that there is an equivalence 
$
\hmfgr(f) \xra{\simeq} \MCM(R) 
$ of triangulated categories 
given by $(F^0, F^1, s^0, s^1) \mapsto \coker(s^0)$. Moreover, it follows from (a graded version of) a theorem of Buchweitz~\cite[Theorem 4.4.1(2)]{buchweitz} that there is an equivalence 
$
\MCM(R) \xra{\simeq} \Dsg(R)
$ of triangulated categories
that sends a module to itself, considered as a complex concentrated in cohomological degree 0. Composing these two functors yields an equivalence of triangulated categories
\begin{equation}
\label{buchweitz}
\coker \c \hmfgr(f) \xra{\simeq} \Dsg(R).
\end{equation}
Notice that $\coker(F(i)) = \coker(F)(i)$ for all $F \in \hmfgr(f)$ and $i \in \Z$. 

\begin{rem}
\label{BGSrem}
Let us explain the relationship between Buchweitz-Greuel-Schreyer's Conjecture \cite[Conjecture A]{BGS} and Conjecture~\ref{BGSconj}. The former states that, in the setting of Conjecture~\ref{BGSconj}, if $M$ is a nonfree MCM $Q/(f)$-module, then $\rank(M) \ge 2^{e-1}$.\footnote{Buchweitz-Greuel-Schreyer assume in \cite[Conjecture A]{BGS} that $M$ has no free summands, but there is no loss in assuming $M$ is simply nonfree. Also, while \cite[Conjecture A]{BGS} does not contain an isolated singularity assumption, Buchweitz-Greuel-Schreyer remark directly below \cite[Conjecture A]{BGS} that the conjecture reduces immediately to the isolated singularity case.} Given a nontrivial matrix factorization $(F^0, F^1,s^0, s^1)$ of $f$, the $Q/(f)$-modules $\coker(s^0)$ and $\coker(s^1)$ are both nonfree and maximal Cohen-Macaulay; since $\rank(F^0) = \rank(\coker(s^0)) + \rank(\coker(s^1))$, \cite[Conjecture A]{BGS} implies Conjecture~\ref{BGSconj}. 
\end{rem}

\subsection{Orlov's Theorem}
\label{sec:orlov}

The following is (a special case of) a celebrated theorem of Orlov:
\begin{thm}[\cite{orlovcoh} Theorem 2.5] 
\label{orlovthm}
There are functors
$$
\Phi_i \co \Dsg(R) \to \Db(X) \quad \text{and} \quad \Psi_i \co \Db(X) \to \Dsg(R)
$$
for all $i \in \Z$ satisfying the following:
\begin{enumerate}
\item If $a > 0$ (i.e. if $X$ is Fano), then $\Phi_i$ is fully faithful, and there is a semiorthogonal decomposition 
$
\Db(X) = \langle \OO(-i -a + 1), \dots, \OO(-i -1), \OO(-i), \Phi_i(\Dsg(R)) \rangle.
$
\item If $a < 0$ (i.e. if $X$ is general type), then $\Psi_i$ is fully faithful, and there is a semiorthogonal decomposition
$
\Dsg(R) = \langle k(-i), k(-i-1) \dots, k(-i+a + 1), \Psi_i(\Db(X)) \rangle.
$
\item If $a = 0$ (i.e. if $X$ is Calabi-Yau), then $\Phi_i$ and $\Psi_i$ are inverse equivalences. 
\end{enumerate}
\end{thm}

Let us recall the formulas for the functors $\Phi_i$ and $\Psi_i$:

\subsubsection*{Definition of the functors $\Phi_i$} We follow Burke-Stevenson's exposition in \cite{BS} of the proof of Orlov's Theorem. The functor $\Phi_i$ factors as a composition
$
\Dsg(R) \xra{\varphi_i} \Dbgr(R) \to \Db(X),
$
where the second map is given by sheafification. Before defining $\varphi_i$, we fix some notation and recall some background. Given a complex $G$ of 
graded free $R$-modules, we let  $G_{\prec i}$ be the subcomplex of $G$ given by free summands of the form $R(i)$ with $i > 0$, and we define $G_{\succeq i} \ce G / G_{\prec i}$. We recall that any complex $C$ of finitely generated graded $R$-modules such that $H^j(C) = 0$ for $j \gg 0$ admits a \emph{minimal free resolution}, i.e. a complex $(F, d_F)$ of graded free $R$-modules such that $F^j = 0$ for $j \gg 0$, $d_F(F) \subseteq \fm F$, and there is a quasi-isomorphism $F \xra{\simeq} C$. Such a resolution exists and is unique by \cite[Proposition 4.4.1]{roberts}.

Suppose $M$ is a bounded complex of finitely generated graded $R$-modules, i.e. an object in $\Dsg(R)$: the object $\varphi_i(M)$ is defined as follows. Let $F$ be the minimal $R$-free resolution of $M$, and let $F'$ be the minimal $R$-free resolution of the complex $\uHom_R(F_{\succeq i}, R)$. Observe that $\uHom_R(F_{\succeq i}, R)$ is a complex of graded free $R$-modules, and it has bounded cohomology since $R$ is Gorenstein; also, $\uHom_R(F_{\succeq i}, R)$ is typically not bounded above, and so it usually does not coincide with $F'$. Finally, we set $\varphi_i(M) \ce \uHom_R(F', R)_{\prec i}$. 

\begin{ex}
Suppose $M$ is a bounded complex of finitely generated graded free $R$-modules; we observe that $\varphi_i(M) = 0$ for all $i$. Indeed, we have $F = M$, and $F' = \uHom_R(M_{\succeq i}, R)$; thus, $\varphi_i(M) = (M_{\succeq i})_{\prec i} = 0$. 
\end{ex}

\begin{ex}
\label{ex:MCM}
Suppose $M$ is an MCM module generated in degree 0; we will discuss explicit examples of such modules in Example~\ref{ex:sharp} below. We now show $\Phi_0(M) = \widetilde{M}$. Let $F$ be the minimal free resolution of $M$, which has the form
$
F = \left[ P \xla{\a} Q \xla{\b}  P(-d) \xla{\a} Q(-d) \xla{\b} \cdots\right]
$
for some graded free $R$-modules $P$ and $Q$ and maps $\a$ and $\b$, where $P$ is a direct sum of copies of $R$. Given a graded $R$-module $N$, write $N^\vee \ce \uHom_R(N, R)$. 
We have $F_{\succeq 0} = F$, and $\uHom_R(F, R) = \left[ P^\vee \xra{\a^\vee} Q^\vee \xra{\b^\vee} P^\vee(d) \xra{\a^\vee} Q^\vee(d) \xra{\b^\vee} \cdots \right]$. The minimal free resolution of $\uHom_R(F, R)$ is the complex $F' = \left[Q^\vee(-d) \xla{\a^\vee} P^\vee(-d) \xla{\b^\vee} Q^\vee(-2d) \xla{\a^\vee} P^\vee(-2d) \xla{\b^\vee} \cdots \right]$. Dualizing $F'$ gives the complex $\left[Q(d) \xra{\a} P(d) \xra{\b} Q(2d) \xra{\a} P(2d) \xra{\b} \cdots \right]$, whose terms are direct sums of copies of $R(j)$ for various~$j > 0$. We conclude that $\varphi_0(M) = \uHom_R(F', R)$, which is quasi-isomorphic to $M$. Thus,  $\Phi_0(M) = \widetilde{M}$. 
\end{ex}

\subsubsection*{Definition of the functors $\Psi_i$} The functor $\Psi_i$ factors as the composition
$
\Db(X) \xra{\psi_i} \Dbgr(R) \to \Dsg(R),
$
where $\psi_i(\cF) = \bigoplus_{j \ge i-a} \mathbf{R}\Gamma(X, \cF(j))$, and the second map is the canonical one~\cite[Remark 2.6]{orlovcoh}. 

\medskip

\begin{prop}
\label{prop:adj}
In the context of Theorem~\ref{orlovthm}, the functor $\Phi_i$ is the left adjoint of the functor $\Psi_{i + a}$. 
\end{prop}

\begin{proof}
Given $i \in \Z$, let $\Db(R)_{\ge i}$ denote the subcategory of $\Db(R)$ given by objects whose cohomology is concentrated in internal degrees at least $i$. It follows from the proof of \cite[Lemma 5.4]{BS} that $\varphi_i$ takes values in $\Db(R)_{\ge i}$. The functor $\Phi_i$ therefore factors as the composition
$
\Dsg(R) \xra{\varphi_i} \Db(R)_{\ge i}  \xra{\pi_i} \Db(X),
$
where $\pi_i$ is given by sheafification. On the other hand, the functor $\Psi_{i+a}$ may be expressed as the composition 
$
\Db(X) \xra{\psi_{i+a}}  \Db(R)_{\ge i} \xra{\pi'_i} \Dsg(R),
$
where $\pi'_i$ is the canonical surjection. The map $\varphi_i$ is the left adjoint of~$\pi'_i$~\cite[Proposition 5.8]{BS}, and $\pi_{i}$ is the left adjoint of~$\psi_{i+a}$; the result follows.
\end{proof}

\subsection{Graded Auslander duality} Let $A = \bigoplus_{i \ge 0} A_i$ be a commutative, graded Gorenstein ring of dimension $t$ such that $A_0 = k$, and assume $A$ has an isolated singularity. There is an isomorphism $\mathbf{R}{\uHom}_A(k, A) \cong k(b)[-t]$ in $\Db_{\gr}(A)$ for some $b, t \in \Z$. For example, our graded hypersurface ring $R$ is Gorenstein, and in this case $b$ is equal to $a = n+1 - d$, and $t$ is equal to $n$. The following result is a version of Auslander duality~\cite{auslander} for the graded singularity category $\Dsing(A)$ (see also \cite[(3.7)]{murfet}). Given a graded $k$-vector space $V$, let $V^* \ce \uHom_k(V, k)$ denote its $k$-dual. 

\begin{prop}[\cite{KMV} Section 4.3] 
\label{auslander}
For any $M, N \in \Dsing(A)$, there is a natural isomorphism 
$$
\Hom_{\Dsing(A)}(M, N)^* \cong \Hom_{\Dsing(A)}(N, M(-b)[t-1]).
$$
That is, the automorphism $M \mapsto M(-b)[t-1]$ of $\Dsing(A)$ is a Serre functor.
\end{prop}

\section{Proof of Theorem~\ref{introthm}}
\label{sec:translation}

 We will need the following technical result. 

\begin{prop}
\label{psik}
 Assume $a \le 0$. Let $\ell \in \Z$, and consider $k(\ell)$ as a complex of graded $R$-modules concentrated in cohomological degree 0. Write $\ell = qd - r$, where $0 \le r < d$. Let $\T_{\PP^n_k}$ denote the tangent bundle on~$\PP^n_k$.  There is an isomorphism $\Phi_0(k(\ell)) \cong i^* (\bigwedge^{r + a} \T_{\PP^n_k})(-r - a)[2q + n  - r - a - 1 ]$ in $\Db(X)$.
\end{prop}

\begin{proof}

We have $F[2] = F(d)$ for all $F \in \hmfgr(f)$, and the equivalence in \eqref{buchweitz} commutes with grading twists and cohomological shifts; it follows that $k(\ell)[2j] \cong k(\ell+jd)$ in $\Dsg(R)$ for all $j \in \Z$. Since the functor $\Phi_0$ commutes with cohomological shifts, we may assume $q = 0$, i.e. $-d  < \ell \le 0$. We therefore must show
$\Phi_0(k(\ell)) \cong i^* (\bigwedge^{a -\ell} \T_{\PP^n_k}(\ell - a))[  n + \ell - a - 1]$.

Let $F$ be the minimal $R$-free resolution of $k$. The resolution $F$ is given by applying the Shamash construction~\cite[Construction 4.1.3]{EP} to the Koszul complex
$$
K \ce \left[S(-n-1) \to S(-n)^{n+1} \to \cdots \to S(-1)^{n+1} \to S \right]
$$
on the variables $x_0, \dots, x_n \in S$. The underlying module of $F$ is thus $K \otimes_S D \otimes_S R$, where $D$ is a divided power algebra on a single variable~$t$ of cohomological degree $-2$ and internal degree $d$. We therefore have:
$$
F^m = \bigoplus_{-s - 2j = m, j \ge 0} S(-s)^{\binom{n+1}{s}} \otimes_S S(-jd) \cdot t^{(j)}\otimes_S R
= \bigoplus_{-s - 2j = m, j \ge 0} R(-s - jd)^{\binom{n+1}{s}}.
$$
The first few terms of $F$ are
$$
\left[ 
\cdots \to R(-3)^{\binom{n+1}{3}} \oplus R(-d - 1)^{n+1}\to R(-2)^{\binom{n+1}{2}} \oplus R(-d) \to R(-1)^{n+1}  \to R  
\right].
$$
We have $F(\ell)_{\succeq 0} = F(\ell)$ (using the truncation notation from the description of the functor $\Phi_0$ in Section~\ref{gradedmf}). Since  $\mathbf{R}$$\uHom_R(k, R) \cong k(a)[-n]$, we have an isomorphism 
$$
\uHom_R(F(\ell)_{\succeq 0}, R) = \uHom_R(F(\ell), R) \cong k(a -\ell)[-n]
$$
in $\Db(R)$, and so the minimal free resolution of  $\uHom_R(F(\ell)_{\succeq 0}, R)$ is $F' \ce F(a-\ell)[-n]$. We now compute:
\begin{align*}
\varphi_0(k(\ell)) & = \uHom_R(F', R)_{\prec  0} \\
&= \uHom_R(F, R)(\ell - a )[n]_{\prec 0} \\
&= \left[ R(\ell -a) \to R(\ell - a + 1)^{n+1} \to R(\ell - a + 2)^{\binom{n+1}{2}} \oplus R(\ell -a + d) \to \cdots \right]_{\prec0},
\end{align*}
where $R(\ell -a)$ is in cohomological degree $-n$. Observe that
$$
\uHom_R(F, R)(\ell -a)[n]  \cong \text{$\mathbf{R}$$\uHom_R(k, R)(\ell -a)[n]$}\cong k(\ell);
$$
in particular, the sheafification of $\uHom_R(F, R)(\ell -a)[n]$ is the zero object in $\Db(X)$. It follows that $\Phi_0(k(\ell))$ is isomorphic to the sheafification of the complex
$$
\left[
R(\ell -a) \to R(\ell -a + 1)^{n+1} \to R(\ell -a + 2)^{\binom{n+1}{2}} \oplus R(\ell - a + d) \to \cdots \right]_{\succeq 0}, \\
$$
where $R(\ell -a)$ is in cohomological degree $-n + 1$. Notice that $\ell - a + d > 0$ (this is where we use the assumption $a \le 0$). Thus, if $\ell > a$, then $\Phi_0(k(\ell)) = 0$; otherwise, $\Phi_0(k(\ell))$ is isomorphic to 
$$
\left[
\OO_X(\ell -a) \to \OO_X(\ell -a + 1)^{n+1} \to \OO_X(\ell -a + 2)^{\binom{n+1}{2}} \to \cdots \to \OO_X^{\binom{n+1}{a - \ell}}
\right],
$$
a brutally truncated Koszul complex on $x_0, \dots, x_n$, where $\OO_X(\ell -a)$ is in cohomological degree $-n+1$. This complex is a locally free resolution of the sheaf
$
i^* \ker\left(\OO_{\PP^n_k}(1)^{\binom{n+1}{a-\ell + 1}} \to \OO_{\PP^n_k}(2)^{\binom{n+1}{a-\ell + 2}} \right) \cong i^* (\bigwedge^{a - \ell} \T_{\PP^n_k})(\ell -a)
$
in cohomological degree $a - \ell -n  +1$, where the isomorphism follows from \cite[Theorem 17.16]{eisenbudCA}.
\end{proof}

\begin{rem}
In the setting of Proposition~\ref{psik}, assume $X$ is smooth. Combining a result of Tu~\cite[Theorem 1.4]{Tu14} and the equivalence~\eqref{buchweitz}, one sees that the category $\Dsing(R)$ is generated by the objects $k, \dots, k(d-1)$. Proposition~\ref{psik} therefore implies that the objects $i^* (\bigwedge^{r} \T_{\PP^n_k})(-r)$ with $0 \le r < d$ generate $\Db(X)$.
\end{rem}

A graded matrix factorization $(F^0, F^1, s^0, s^1)$ is called \emph{reduced} if the maps $s^0$ and $s^1$ are minimal, i.e. $s^0 \otimes k = s^1 \otimes k = 0$. 

\begin{thm}
\label{translation}
Assume $X$ is smooth and $a \le 0$, and let $\cC \in \Db(X)$. Let $F = (F^0, F^1, s^0, s^1)$ be a reduced graded matrix factorization of $f$ such that $\Psi_a(\cC) \cong \coker(F)$, where $\Psi_a$ is as in Theorem~\ref{orlovthm}. For $i \in \{ 0,1\}$ and $j \in \Z$, write $b^i_j \ce \dim_k \Hom_S(F^i, k(-j))$, the number of degree~$j$ generators of $F^i$. 
\begin{enumerate}
\item Fix $i \in \{0,1\}$ and $j \in \Z$, and write $a-j = qd - r$ for $0 \le r < d$. We have: 
$$
b^i_j = h^{r + a -2q - i + 1}\left(\PP^n_k, i_*(\cC) \otimes \Omega^{r+a}_{\PP^n_k}(r+a)\right).
$$
\item Additionally, $\rank(F) = \rho(\cC)$, where $\rho$ is as defined in~\eqref{eqn:rho}. 
\end{enumerate}
\end{thm}

\begin{proof}
Write $M \ce \coker(F)$. Since $F$ is reduced, the minimal free resolution of $M$ has the form:
$$
\cdots \xra{s^1} F^0(-d) \xra{s^0} F^1(-d) \xra{s^1} F^0 \xra{s^0} F^1 \to M.
$$
Thus, the graded Betti number of $M$ in cohomological degree 0 (resp. $-1$) and internal degree $j$ is $b^1_j$ (resp. $b^0_j)$. By \cite[Lemma 3.1]{pavlov21} and the identification of the shift $M[1]$ in $\Dsg(R)$ with the first syzygy of $M$, we have $b^i_j = \dim_k \Hom_{\Dsg(R)}\left(\Psi_a(\cC), k(-j)[i-1]\right)$. We now compute:
\begin{align*}
\Hom_{\Dsg(R)}\left(\Psi_a(\cC), k(-j)[i-1]\right) &=  \Hom_{\Dsg(R)}\left(k(a-j), \Psi_a(\cC)[n-i]\right) & (\text{Proposition}~\ref{auslander})\\
 &= \Ext^{n-i}_X\left(\Phi_0(k(a-j)), \cC\right) & (\text{Proposition}~\ref{prop:adj})\\
&=  \Ext_X^{r + a + 1 -i - 2q}\left(i^* (\bigwedge^{r+a} \T_{\PP^n_k})(-r-a), \cC \right) & (\text{Theorem}~\ref{psik}). 
\end{align*}
Finally, we have:
\begin{align*}
\Ext_X^{r + a + 1 -i - 2q}\left(i^* (\bigwedge^{r+a} \T_{\PP^n_k})(-r-a), \cC \right)  &= H^{r + a + 1 -i - 2q}\left(X,  \cC \otimes i^* (\O^{r+a}_{\PP^n_k})(r+a) 
\right)\\ 
&= H^{r + a + 1 -i - 2q}\left(\PP^n_k, i_*(\cC) \otimes \O^{r+a}_{\PP^n_k}(r+a) 
 \right),
\end{align*}
where the second equality follows from the projection formula. This proves (1). For~(2), we note that 
$
\rank(F) = \rank(F^0) + \rank(F^1) 
= \sum_{j \in \Z} (b^0_j + b^1_j)
~= \rho(\cC).
$
\end{proof}

\begin{proof}[Proof of Theorem~\ref{introthm}]
Assume the graded version of Conjecture~\ref{BGSconj} holds for $f$. Let $\cC \in \Db(X)$. Choose $F = (F^0, F^1, s^0, s^1) \in \mfgr(f)$ such that $\coker(F) \cong~\Psi_a(\cC)$. We have 
$
2^{e+1} \le 2 \cdot \rank(F^0) =\rho(\cC), 
$
where the equality follows from Theorem~\ref{translation}(2). 
This proves Part (1). Assume Conjecture~\ref{newconj} holds, let $F$ be a nontrivial graded matrix factorization of $f$, and let $\cC \ce ( \Phi_0 \circ \coker)(F)$. By Theorem~\ref{orlovthm}(3) and Theorem~\ref{translation}(2), we have 
$
2 \cdot \rank(F^0) = \rho(\cC) \ge 2^{e+1}.
$
\end{proof}

\begin{rem}
\label{beilinsonremark}
Given any object $\cC \in \Db(X)$, there is a second-quadrant spectral sequence 
$$
E_1^{p, q} = H^q\left(\PP^n_k, i_*(\cC) \otimes \Omega_{\PP^n_k}^{-p}(-p)\right) \otimes_k \OO(p)  \Rightarrow \mathcal{H}^{p+q}(i_*(\cC))
$$
called the \emph{Beilinson spectral sequence} \cite{beilinson}. Notice that the total rank of the $E_1$ page of this spectral sequence is exactly $\rho(\cC)$. 
\end{rem}

\section{Examples}

We first note that the invariant $\rho$ enjoys a duality property:

\begin{prop}
\label{lem:dual}
For any bounded complex $\cC$ of locally free $\OO_X$-modules, we have $\rho(\cC) = \rho(\cC^\vee(1-a))$, where $\cC^\vee \ce \mathcal{H}om_X(\cC, \OO_X)$.
\end{prop}

\begin{proof}
Noting that the canonical bundle of $X$ is $\OO_X(-a)$, we have:
\begin{align*}
\rho(\cC) &\ce \sum_{r = 0}^n h\left(\PP^n_k, i_*(\cC) \otimes \Omega^r_{\PP^n_k}(r) \right)\\
&= \sum_{r = 0}^n h\left(\PP^n_k, i_*(\cC)^\vee \otimes \bigwedge^r\cT_{\PP_k^n}(-r-a) \right) \\
&= \sum_{r = 0}^n h\left(\PP^n_k, i_*(\cC)^\vee \otimes \Omega_{\PP^n_k}^{n-r}(n+1-r-a) \right) \\
&= \sum_{r = 0}^n h\left(\PP^n_k, i_*(\cC)^\vee(1-a) \otimes \Omega_{\PP^n_k}^{r}(r) \right),
\end{align*}
where the second equality follows from Serre duality, the third from the isomorphism $\bigwedge^{r}\cT_{\PP_k^n} \cong \Omega_{\PP^n_k}^{n-r}(n+1)$ for all $0 \le r \le n$ (see, for instance, the proof of \cite[Theorem 17.16]{eisenbudCA}), and the last from reindexing. 
\end{proof}

We next observe that Conjecture~\ref{newconj} holds for points and for the structure sheaf $\OO_X$.
\begin{ex}
\label{ex:point}
Let $\cF$ be the structure sheaf of a point on $X$, and assume $n \ge 1$. We evidently have $$\rho(\cF) = \sum_{r = 0}^n \rank (\Omega^r_{\PP^n_k}) = 2^n  \ge 2^{e + 1}.$$
\end{ex}

\begin{ex}
By~\cite{OSS} Chapter 1, Section 1.1], for all $0 \le p, q \le n$ and $\ell \in \Z$, we have:
\begin{equation}
\label{lem:omega}
h^q(\PP^n_k, \Omega^p_{\PP^n_k}(\ell)) = 
\begin{cases}
\binom{\ell + n - p}{\ell} \binom{\ell-1}{p},  & q = 0, \text{ }\ell > p; \\
1, & \ell = 0, \text{ } q = p; \\
\binom{p - \ell}{-\ell}\binom{-\ell-1}{n-p}, & q = n, \text{ }\ell < p - n;\\
0, & \on{otherwise.}
\end{cases}
\end{equation}
For each $0 \le r \le n$, there is a short exact sequence
\begin{equation}
\label{eqn:ses}
0 \to \Omega^r_{\PP^n_k}(r  -d) \to \Omega^r_{\PP^n_k}(r) \to \OO_X\otimes \Omega^r_{\PP^n_k}(r) \to 0.
\end{equation}
Assume now that $a \le 0$. By~\eqref{lem:omega}, we have:
\begin{enumerate}
\item $h^i\left(\PP_k^n, \O_{\PP_k^n}^r(r  - d)\right) = 
\begin{cases} 
\binom{d}{d - r} \binom{d-r-1}{n-r},  & i = n; \\
0, & i \ne n.
\end{cases}
$
\item $h^i\left(\PP_k^n, \O_{\PP_k^n}^r(r )\right) = 
\begin{cases} 
1, & i = r = 0; \\
0, & \text{else.}
\end{cases}$
\end{enumerate}
The short exact sequence \eqref{eqn:ses} therefore implies:
$$
\rho(\OO_X) = 1 + \sum_{r = 0}^n \binom{d}{d - r}\binom{d-r-1}{n-r} \ge 1 + \sum_{r = 0}^n \binom{n+1}{n+1 - r} =
2^{n+1}  >2^{e + 1}.
$$

\end{ex}

\begin{rem}
\label{rem:false}
The assumption in Conjecture~\ref{newconj} that $a \le 0$ is necessary. For instance, suppose $n \ge 2$ and that $f$ is a linear form, so that $a = n  > 0$. For all $1 - n \le j \le 0$, the calculation~\eqref{lem:omega} and the short exact sequence~\eqref{eqn:ses} imply that
$$
\rho(\OO_X(j)) = h^{-j}(\PP_k^n,\OO_X \otimes \Omega^{-j}_{\PP^n_k}) + h^{-j}(\PP_k^n,\OO_X \otimes \Omega^{1-j}_{\PP^n_k}(1))  = h^{-j}(\PP_k^n,\Omega^{-j}_{\PP^n_k}) + h^{1-j}(\PP_k^n,\Omega^{1-j}_{\PP^n_k}) = 2 < 2^{e + 1}.
$$
Thus, each of the line bundles $\OO_X(1-n), \dots, \OO_X$ violates the inequality in Conjecture~\ref{newconj}. Additionally, if $n = d = 2$, then $\rho(\OO_X) = 2 < 2^{e+1} = 4$.

As a concrete example, take $n = 2$, $j = -1$, and $d = 1$. Formula~\ref{lem:omega} implies that
$$
h^i(\PP^n_k, \Omega^r_{\PP^n_k}(r-2)) = 
\begin{cases}
1, & r = i = 2, \\
0, & \text{else};
\end{cases}
\quad \text{and} \quad
h^i(\PP^n_k, \Omega^r_{\PP^n_k}(r-1)) =
\begin{cases}
1, & r = i = 1, \\
0, & \text{else}.
\end{cases}
$$
It thus follows from the short exact sequence \eqref{eqn:ses} (twisted by $-1$) that
$$
\rho(\OO_X(-1)) = h^1(\PP^n_k, \OO_X \otimes \Omega^1_{\PP^n_k})  + h^1(\PP^n_k, \OO_X \otimes \Omega^2_{\PP^n_k}(1)) = h^1(\PP^n_k, \Omega^1_{\PP^n_k}) + h^2(\PP^n_k, \Omega^2_{\PP^n_k}) = 2.
$$ 
A nearly identical calculation shows that, when $n = d = 2$, we have $\rho(\OO_X) = 2$.

\end{rem}

Finally, we consider a family of examples where the bound in Conjecture~\ref{newconj} is achieved.
\begin{ex}
\label{ex:sharp}
Assume $k$ contains a square root $i$ of $-1$, and suppose $f = x_0^n + \cdots + x_n^n$, where $n$ is even. In this case, $e = \frac{n}{2}$. We have graded matrix factorizations
$$
F = \left(S(-e), S, x_0^{e}, x_0^{e}\right) \in \mfgr(x_0^n), \quad F' = \left(S(-e), S, x_0^{e} + ix_1^{e},  x_0^{e} - ix_1^{e}  \right) \in \mfgr(x_0^n + x_1^n).
$$
We may construct a graded matrix factorization of $f$ whose rank meets the bound in Conjecture~\ref{BGSconj} by taking tensor products of $F$ and $F'$; we refer the reader to \cite[Section 4]{BD} for the definition of the tensor product of graded matrix factorizations. In more detail:
the tensor product of one copy of $F$ and $e$ copies of $F'$ gives a reduced graded matrix factorization $G = (G^0, G^1, s^0, s^1)$ of $f$, where $G^0 = S(-e)^{2^e}$, and $G^1 = S^{2^e}$. Let $M \ce \coker(s^0)$; since $M$ is MCM and generated in degree 0, $\Phi_0(M)$ is the vector bundle $\widetilde{M}$ (Example~\ref{ex:MCM}). Since $a = 0$, we conclude from Theorem~\ref{orlovthm}(3) that $\Psi_0(\widetilde{M}) = M$. By Theorem~\ref{translation}(2), we therefore have $\rho(\widetilde{M}) = 2^{e+1}$, i.e. $\widetilde{M}$ achieves the bound in Conjecture~\ref{newconj}. 
\end{ex}

\bibliographystyle{amsalpha}
\bibliography{Bibliography}

\end{document}